\documentclass[12pt,reqno]{amsart}
\usepackage{graphicx,subfig,mathrsfs}
\input RCT.sty
\def\ttx{{\tt x}}

\title[Kernel Hierarchy]{A hierarchical structure of transformation semigroups with applications to probability limit measures}

\author{G. Budzban and Ph. Feinsilver}
\address{Department of Mathematics\\
Southern Illinois University\\
Carbondale, IL\\
62901 USA}

\begin{document} 
\maketitle
\begin{abstract}
The structure of transformation semigroups on a finite set is analyzed by introducing a hierarchy of functions mapping subsets to subsets.  
The resulting hierarchy of semigroups has a corresponding hierarchy of minimal ideals, or kernels.  This kernel hierarchy produces a set of tools that provides direct access to computations of interest in probability limit theorems; in particular, finding certain factors of idempotent limit measures. 
In addition, when considering transformation semigroups that arise naturally from edge colorings of directed graphs, as in the road-coloring problem, the hierarchy produces simple techniques to determine the rank of the kernel and to decide when a given kernel is a right group.  
In particular, it is shown that all kernels of rank one less than the number of vertices must be right groups and their structure for the case of two generators is described.
\end{abstract}
\thispagestyle{empty}

\begin{section}{Introduction}
The study of transformation semigroups has applications in fields as diverse as probability, automata theory, and discrete dynamical systems.  There are two main sources of this current investigation: discrete dynamical systems that emerge from edge-coloring digraphs, such as in the road-coloring problem and limit theorems for convolution powers of probability measures on finite semigroups.  In all of these cases, it is important to discern asymptotic properties of the system 
directly from the generators.  Techniques that permit this forecasting of asymptotic behavior have many applications according to the context. \bigskip

To this end, in Section 2 of this paper, given a starting collection of functions on a finite set, of size $n$, 
we consider the hierarchy of functions induced on subsets of level $\l$, for $1\le\l\le n$, and develop the machinery for this hierarchy.  
In Section 3, we recall the structure of finite semigroups and of idempotent probability measures on these semigroups.  In Section 4, the kernel hierarchy is analyzed in detail and an important theorem of Friedman is recalled.  In addition, the hierarchy is used to obtain information about the limit measures of Section 3.  
Finally, in section 5, hierarchy computations which reveal the rank and structure of the kernel are presented.
\end{section}

\begin{section}{Function hierarchy}

\begin{notation}     
Indices, even though denoted as $\ell$-tuples are identified with the corresponding {\it sets\/}, i.e., unordered, distinct elements,
denoted using uppercase roman letters. As indices, they follow \textit{dictionary order} within each level $\l$.  \\
All vectors are conventionally row vectors. Transposes will be denoted by superscript plus, so
that a typical column vector is $v^+$. We use italic type for the identity matrix. 
\end{notation} \bigskip

Given any finite set $S$, of cardinality $|S|=n$,
let $\F=\{f\colon S\to S\}$. Let $\P=\P(S)$ denote the power set of $S$. Denote the $\l^{\rm th} $\textit{layer}, 
$$\P_\l=\{\rI:|\rI|=\l\}\subset \P$$ of subsets of cardinality $\l$, $0\le\l\le n$. We wish to extend functions in $\F$ to
$$\P=\bigcup_{\l=0}^n \P_\l$$
acting on layers. For layer 0, define $f(\emptyset)=\emptyset$. For layer $\l$, we adjoin the empty set and define
$$\spow{f}{\l}(\rI)=\begin{cases} \rJ,&\text{ if } f(\rI)=\rJ\in\P_\l\\ \emptyset,&\text{ otherwise}\end{cases}
$$
We have a correspondence of functions with matrices $f\leftrightarrow F$, $\spow{f}{\l}~\leftrightarrow~\spow{F}{\l}$ accordingly:
$$ F_{ij}=\begin{cases}1,&\text{ if } f(i)=j\\ 0,&\text{ otherwise}\end{cases}$$
And for $\l\ge1$,
$$ \spow{F}{\l}_{\rI\rJ}=\begin{cases}1,&\text{ if } f(\rI)=\rJ\in\P_\l\\ 0,&\text{ otherwise}\end{cases}$$
Note that we adjoin the empty set in the definition of $\spow{f}{\l}$ in order to consider composition of functions on $\P_\l$.
Using the corresponding matrix, we see that a row of zeros is automatically perpetuated when taking matrix products. \bigskip

In order to get a non-zero entry in \spow{F}{\l}, we observe that since $\rI$ and $\rJ$ have cardinality $\l$, we can form the
$\l\times\l$ submatrix of $F$ with rows $\rI$ and columns $\rJ$. Since $\rI$ maps onto $\rJ$, this must be a permutation matrix.
Since we want an entry of $+1$ regardless of the sign of the permutation, we are taking the \textit{permanent} of this submatrix.
In other words,
\begin{quotation}
the $\rI\rJ$ entry of \spow{F}{\l} is the permanent of the submatrix of $F$ with row labels $\rI$ and column labels $\rJ$
\end{quotation}
With this interpretation, the zero rows appear quite naturally. \bigskip

The main definitions are  \bigskip

\begin{definition} An $\l$-set $\rI$ is said to be $f$-\textit{preserved} if it maps to $\rJ\in\P_\l$ by $f$. \medskip
We say that $f$ \textit{collapses} $\rI$ if $|f(\rI)|<|\rI|$.
\end{definition} 

So the $f$-preserved sets provide the non-zero entries in the corresponding power $\spow{f}{\l}$. \bigskip

\begin{subsubsection}{\bf Notation}\label{sssec:not}
 A convenient notation for functions on an $n$-set is to write a label of the form 
$$[f(1)f(2)\ldots f(n)]$$
For example, $f=[2344]$ means that $1\to2\to3\to4$, $4\to4$, by $f$.
\end{subsubsection} \bigskip

\begin{example} Let $n=4$. Then $f=[2344]$ has matrix
$$F=\begin{bmatrix}0&1&0&0\\ 0&0&1&0 \\ 0&0&0&1\\ 0&0&0&1\end{bmatrix}$$
For $\l=2$, we have row and column labels $12,13,14,23,24,34$, corresponding to the subsets of size 2, and
$$\spow{F}{2}= \left[ \begin {array}{cccccc} 0&0&0&1&0&0\\0&0&0&0&
1&0\\0&0&0&0&1&0\\0&0&0&0&0&1
\\0&0&0&0&0&1\\0&0&0&0&0&0\end {array} \right]$$
Consider $g$ with matrix
$$G=\begin{bmatrix}0&1&0&0\\ 0&1&0&0 \\ 0&0&0&1\\ 0&0&1&0\end{bmatrix}$$
and 
$$\spow{G}{2}= \left[ \begin {array}{cccccc} 0&0&0&0&0&0\\0&0&0&0&1&0\\
0&0&0&1&0&0\\0&0&0&0&1&0\\0&0&0&1&0&0\\0&0&0&0&0&1\end {array} \right]$$

We check that the composition $h=g^\circ f$ corresponds to matrix multiplication $H=FG$:
$$\spow{H}{2}= \left[ \begin {array}{cccccc} 0&0&0&0&1&0\\0&0&0&1&0&0\\0&0&0&1&0&0\\0&0&0&0&0&1
\\0&0&0&0&0&1\\0&0&0&0&0&0\end {array} \right]$$
which is easily checked to agree with $\spow{F}{2}\spow{G}{2}$.
\end{example}
We now prove the main result of this section. 

\begin{theorem}\label{thm:hom}
For each $\l\ge1$, the map $f\to\spow{f}{\l}$ is a homomorphism of the semigroup of
functions under composition. Equivalently, in terms of the correspondence with matrices, we have the matrix homomorphisms
$$\spow{(FG)}{\l}=\spow{F}{\l}\spow{G}{\l}$$
\end{theorem}
\begin{proof}
Consider the $\rI\rJ$ entry in the product $FG$:
$$(\spow{F}{\l}\spow{G}{\l})_{\rI\rJ}=\sum_\rK F_{\rI\rK}G_{\rK\rJ}=F_{\rI f(\rI)}G_{f(\rI)g(f(\rI))}$$
This yields the value one if and only if $\rI$ is $f$-preserved and $f(\rI)$ is $g$-preserved.
While
$$\spow{(FG)}{\l}_{\rI\rJ}=\begin{cases}1,&\text{ if } g(f(\rI))=\rJ\in\P_\l\\ 0,&\text{ otherwise}\end{cases}$$
If $f$ collapses $\rI$, then $|g(f(\rI))|\le |f(\rI)|<|\rI|=\l$. And if $g$ collapses $f(\rI)$, we immediately have $\spow{(FG)}{\l}_{\rI\rJ}=0$.
So we get a 1 precisely as required.
\end{proof}
 
\begin{subsection}{Local inclusion operators}
For layers $\l$ and $m>\l$, define the inclusion operator
$$\EP{E}{\l, m}=\begin{cases}1,&\text{ if } \P_\l\ni\rI\subset\rJ\in\P_m\\ 0,&\text{ otherwise}\end{cases}$$
And, for $\l<m$,
$$\EP{E}{m,\l}=(\EP{E}{\l, m})^\top$$
the $\top$ indicating matrix transpose. In particular,
$$\EP{E}{\l,\l-1}=\begin{cases} 1,& \rI\supset\rJ \text{ and }|\rI\setminus\rJ|=1,|\rI|=\l\\ 0,&\text{ otherwise}\end{cases}$$
It is convenient to write this as a condition, i.e.,  $\chi(\rI\supset\rJ)$ to denote the above entries.
\begin{theorem}\label{thm:linc}
For any function $f$, if $\rI$ is an $f$-preserved set, then, for $\l>1$,
$$(\EP{E}{\l,\l-1}\spow{F}{(\l-1)})_{\rI\rJ}=(\spow{F}{\l}\EP{E}{\l,\l-1})_{\rI\rJ}$$
\end{theorem}
\begin{proof}
We have $\rI$ such that $|f(\rI)|=|\rI|$. Now,
$$(\spow{F}{\l}\EP{E}{\l,\l-1})_{\rI\rJ} =\sum_\rK F_{\rI\rK}\chi(\rK\supset\rJ)=
\begin{cases} 1,& f(\rI)\supset\rJ, |f(\rI)\setminus\rJ|=1 \\ 0,&\text{ otherwise}\end{cases}$$
since $(\spow{F}{\l})_{\rI,*}$ is a row vector with precisely one non-zero entry: $\spow{F}{\l}_{\rI,f(\rI)}$.
Suppose that 
$$(\EP{E}{\l,\l-1}\spow{F}{(\l-1)})_{\rI\rJ}=\sum_\rK \chi(\rI\supset\rK)F_{\rK\rJ}=1\ .$$
Then $\rI\supset\rK \xrightarrow{f}\rJ\in\P_{\l-1}$. Therefore, $\rJ\subset f(\rI)$, with $|\rJ|=\l-1$, so
$(\spow{F}{\l}\EP{E}{\l,\l-1})_{\rI\rJ}=1$.\medskip

Conversely, suppose $(\spow{F}{\l}\EP{E}{\l,\l-1})_{\rI\rJ}=1$. 
That is, $f(\rI)\supset\rJ=f(\rK)$, for some $\rK\subset \rI$. 
Suppose the entry $(\EP{E}{\l,\l-1}\spow{F}{(\l-1)})_{\rI\rJ}$ is greater
than 1. Then there exist subsets $\rK_i$, $i=1, 2$, such that $\rI\supset\rK_i$ and $f(\rK_i)=\rJ$. Now, 
$|\rK_1|=|\rK_2|=\l-1$. Thus, if $\rK_1\ne\rK_2$, it must be that $\rK_1\cup\rK_2=\rI$, otherwise they are each missing the same element.
Then $f(\rI)=\rJ$, with $|\rJ|=\l-1$, so $\rI$ could not be $f$-preserved. We conclude that $(\EP{E}{\l,\l-1}\spow{F}{(\l-1)})_{\rI\rJ}$
is at most one for any $f$-preserved set, $\rI$. Thus, $(\EP{E}{\l,\l-1}\spow{F}{(\l-1)})_{\rI\rJ}=1$ as required.
\end{proof}
\begin{example} We continue the previous example with $f=[2344]$. For $\l=2$, 
$$\EP{E}{2,1}= \left[ \begin {array}{cccc} 1&1&0&0\\1&0&1&0\\
1&0&0&1\\0&1&1&0\\0&1&0&1\\0&0&1&1\end {array}
 \right]$$
and we have

$$\spow{F}{2}\EP{E}{2,1}=
\left[ \begin {array}{cccc} 0&1&1&0\\0&1&0&1\\0&1&0&1\\0&0&1&1\\0&0&1&1\\0&0&0&0\end {array} \right]\,,\quad
\EP{E}{2,1}F= 
\left[ \begin {array}{cccc} 0&1&1&0\\0&1&0&1\\0&1&0&1\\0&0&1&1\\0&0&1&1\\0&0&0&2\end {array} \right]$$
with $\{3,4\}$ collapsing.
\end{example}
\end{subsection}
\begin{subsection}{Vector hierarchy}
We start with some definitions, conventions:\medskip

1. Let $\V$ be the vector space over $\RR$ with orthonormal basis $e_\rI$, as $\rI$ runs through $\P \setminus \emptyset$.
Given $\l\ge1$, $\V_\l$ is the span of $e_\rI$ with $\rI\in\P_\l$. \bigskip

We think of a vector $v\in\V$ as a ``field", $v(\rI)$, defined on $\P \setminus \emptyset$. 
The vector hierarchy relates the various levels of the field $v(\rI)$ in terms of its homogeneous components,
$\displaystyle v_\l=\sum_{|\rI|=\l} v_\l(\rI) e_\rI$, as $\rI$ moves through the layers. \bigskip

2. A vector $v_\l\in\V_\l$ is \textit{$f$-compatible} if $v_\l=\sum_{\rI\in\P_\l}v_\l(\rI) e_\rI$ has ``$f$-preserved support", i.e.,
$v_\l(\rI)\ne0$ implies $\rI$ is $f$-preserved.  In global terms, we could say the field $v(\rI)$ is ``$f$-compatible" if it is
supported on $f$-preserved sets in every layer $\l>0$.\bigskip

Now for some basic observations which are immediate. We denote the rank of $f$, $r=|f(S)|$.  \bigskip

\begin{proposition} Let $f\in\F$ have rank $r$. Then:\medskip

1. $\displaystyle \mathrm{rank\,}(\spow{f}{\l)})=\binom{r}{\l}$. \bigskip

2. $\spow{f}{\l}=0$ if $\l>r$. \bigskip

3. $\rI$ is $f$-preserved if and only if $\rJ$ is $f$-preserved for every non-empty subset $\rJ \subset \rI$.
\end{proposition}
The hierarchy is constructed by successive applications of the operators $\EP{E}{\l,\l-1}$.
\begin{lemma}\label{lem:comp} Let $v_\l\in\V_\l$ be $f$-compatible. Then $v_{\l-1}=v_\l\,\EP{E}{\l,\l-1}$ is $f$-compatible.
\end{lemma}
\begin{proof} Write
$$v_{\l-1}(\rI)=\sum_\rJ v_\l(\rJ) \,\chi(\rJ \supset \rI) $$
Since $v_\l$ is supported on $f$-preserved sets,  terms in the sum can only be nonzero on sets $\rI$ that are subsets of $f$-preserved sets, hence themselves
$f$-preserved.
\end{proof}
And the main property
\begin{theorem} \label{thm:hier}
Let $v_\l$ be an $f$-compatible left eigenvector of $\spow{F}{\l}$, $v_\l\spow{F}{\l}=\lambda\,v_\l$. Then
$v_{\l-1}=v_\l\,\EP{E}{\l,\l-1}$ is an $f$-compatible left eigenvector of $\spow{F}{(\l-1)}$ with the same eigenvalue $\lambda$.
\end{theorem} \bigskip

\begin{remark} This says you can move down the hierarchy. Moving up with right eigenvectors does not generally work.
For example, take a stochastic matrix, $F$, at level 1 that collapses \{1,2\}, say. Then the all-ones vector $u$ is a right eigenvector,
which will map to a multiple of an all-ones vector at level 2 by $\EP{E}{2,1}$, while the top row of $\spow{F}{2}$ has all zeros.
\end{remark} \bigskip

\begin{proof}
We noted in Lemma \ref{lem:comp} that $v_{\l-1}$ is $f$-compatible. Now
\begin{align*}
\sum_{\rK} v_{\l-1}(\rK)\spow{F}{(\l-1)}_{\rK\rI}&=\sum_{\rK,\rJ}v_\l(\rK)\EP{E}{\l,\l-1}_{\rK\rJ}\spow{F}{(\l-1)}_{\rJ\rI}\\
&=\sum_{\rK,\rJ}v_\l(\rK)\spow{F}{\l}_{\rK\rJ}\EP{E}{\l,\l-1}_{\rJ\rI} &&\text{ by Theorem \ref{thm:linc}}\\
&=\lambda\,\sum_{\rJ}v_\l(\rJ)\EP{E}{\l,\l-1}_{\rJ\rI}=\lambda v_{\l-1}(\rI)
\end{align*}
The theorem applies since $v_\l$ is $f$-compatible so that only rows,  the $\rK$-indices in the sums, that are $f$-preserved appear with
non-zero factors.
\end{proof}
\begin{corollary}[Hierarchy] \label{cor:hch}
Starting with an $f$-compatible left eigenvector $v_r\in\V_r$, the hierarchy of vectors
$$v_r, v_{r-1}=v_r\EP{E}{r,r-1} ,\ldots,v_\l=v_{\l+1}\EP{E}{\l+1,\l} ,\ldots, v_1 =v_2\EP{E}{2,1}$$
are $f$-compatible left eigenvectors in their respective layers. They are all with the same eigenvalue.
\end{corollary}
In this sense the ``hierarchy" is a field, $v(\rI)=v_\l(\rI)$, on layer $\l$, defined on $\P\setminus\emptyset$.
\end{subsection}
\begin{subsection}{Inclusion operators and levels}
It is clear that the successive application of the operators $\EP{E}{\l,\l-1}\EP{E}{\l-1,\l-2}$ is equivalent to
the single operator $\EP{E}{\l,\l-2}$, and so on. For example, counting orderings, one checks the following.

\begin{align}\label{eq:eops}
\EP{E}{b,b-1}\EP{E}{b-1,b-2}\cdots \EP{E}{a+1,a}&=(b-a)!\,\EP{E}{b,a}\\
\EP{E}{m,m-1}\EP{E}{m-1,m-2}\cdots \EP{E}{2,1}&=(m-1)!\,\EP{E}{m,1}
\end{align}
Working with vectors, we can always rescale and use successive products or direct inclusion between levels as convenient. \bigskip

An interesting situation to consider is that $\EP{E}{\l,n-\l}$ is a square matrix. In fact, it is invertible.  Following \cite{BA},
define \textit{exclusion operators}, $\EP{\bar E}{\l,m}$ by
$$(\EP{\bar E}{\l,m})_{\rI\rJ}=\begin{cases}1,&\text{ if } \rI\cap\rJ=\emptyset\\ 0,&\text{ otherwise}\end{cases}$$
Then
\begin{proposition}\cite[eq.\,(1)]{BA} \\
For $1\le\l\le \lfloor n/2\rfloor$,
$$
(\EP{E}{\l,n-\l})^{-1}=\sum_{i=0}^\l (-1)^i \binom{n-i-\l}{\l-i}^{-1} \EP{\bar E}{n-\l,i}\EP{E}{i,\l}\ .
$$
\end{proposition}
\begin{example} For $n=4$, we have the inverse pair
$$\EP{E}{1,3}=  \left[ \begin {array}{rrrr} 1&1&1&0\\1&1&0&1
\\1&0&1&1\\0&1&1&1\end {array}
 \right]\,,\qquad (\EP{E}{1,3})^{-1}= \left[ \begin {array}{rrrr} 1/3&1/3&1/3&-2/3\\1/3&1
/3&-2/3&1/3\\1/3&-2/3&1/3&1/3\\-2/
3&1/3&1/3&1/3\end {array} \right]
$$
and for $n=5$,
$$\EP{E}{1,4}=   \left[ \begin {array}{rrrrr} 1&1&1&1&0\\1&1&1&0&1
\\1&1&0&1&1\\1&0&1&1&1
\\0&1&1&1&1\end {array} \right]\,,\quad
(\EP{E}{1,4})^{-1}=\left[ \begin {array}{rrrrr} 1/4&1/4&1/4&1/4&-3/4\\
1/4&1/4&1/4&-3/4&1/4\\1/4&1/4&-3/4&1/4&1/4
\\1/4&-3/4&1/4&1/4&1/4\\-3/4&1/4&1
/4&1/4&1/4\end {array} \right]
$$
\end{example}
Notice the pattern for $\EP{E}{1,n-1}$. Clearly, given $i$, there is only one $(n-1)$-set not containing it. By dictionary ordering,
it occurs on the antidiagonal. We can easily verify the inverse. \bigskip

\begin{proposition}\label{prop:enminusone}
 Let $I$ denote the $n\times n$ identity matrix, $J$, the $n\times n$ all-ones matrix, and
$I'$ the $n\times n$ matrix with 1's on the antidiagonal, zeros elsewhere. Then \medskip

1. $\EP{E}{1,n-1}=J-I'$.  \bigskip

2. $(\EP{E}{1,n-1})^{-1}=\frac{1}{n-1}\,J-I'$.
\end{proposition}
\begin{proof} Note that $JI'=I'J=J$. Now, using $J^2=nJ$, $I'^2=I$, we have
$$(J-I')(J-(n-1)I')=J^2-I'J-(n-1)JI'+(n-1)I=(n-1)I$$
dividing through by $n-1$ yields the result.
\end{proof}
Thus, we have the up-down symmetry of the hierarchies. Combining the general properties with Corollary \ref{cor:hch}, we have
\begin{proposition} \hfill\break \label{prop:nmo}
1. Let $1\le\l\le n/2$. The mappings $v_\l=v_{n-\l}\EP{E}{n-\l,\l}$, $v_{n-\l}=v_\l(\EP{E}{n-\l,\l})^{-1}$,
provide a  1-1 correspondence between $\V_\l$ and $\V_{n-\l}$. \bigskip

2. If $f$ has rank $r>n/2$, then the inclusion maps yield a 1-1 correspondence between eigenvectors in 
layers $\l$ and $n-\l$, for $n-r\le\l\le  n/2$.
\end{proposition}
We have the somewhat surprising fact that vectors in a lower layer determine those near the top, even though the hierarchy is constructed
from top down.
\end{subsection}
\end{section}

\begin{section}{Semigroups and kernels}\label{sec:sgrpskrnls}
A finite semigroup, in general, is a finite set, $\SG$, which is closed under a binary associative operation. 
For any subset $T$ of $\SG$, we will write $\E(T)$ to refer to the set of idempotents in $T$. 
\begin{theorem}\cite[Th. 1.8]{HM}	\label{thm:kernel} 
Let $\SG$ be a finite semigroup. Then $\SG$ contains a minimal ideal $\K$
called the kernel which is a disjoint union of isomorphic groups. 
In fact, $\mathcal{K}$ is isomorphic to $\ms.X. \times \ms.G.  \times \ms.Y.$ where, given $e\in \E(\SG)$, then
$e\mathcal{K}e$ is a group and
\[
\ms.X. =\E( \K e)\,,\qquad  \ms.G. = e\K e\,, \qquad  \ms.Y. = \E(e\K)
\]
and if $\left( x_1,g_1,y_1\right) $, $\left( x_2,g_2,y_2\right)  \in \ms.X. \times \ms.G. \times \ms.Y.$
then the multiplication rule has the form
\[
(x_1, g_1, y_1) (x_2, g_2, y_2) = (x_1, g_1\phi( y_1,x_2) g_2, y_2)
\]
where $\phi\colon \ms.Y.\times\ms.X.\to\ms.G.$ is the \textit{sandwich function}.
\end{theorem}
The product structure $\ms.X. \times \ms.G. \times \ms.Y.$ is called a Rees product and any semigroup that has 
a Rees product is \textit{completely simple}. 
The kernel of a finite semigroup is always completely simple.

\subsection{Kernel of a matrix semigroup}  \label{ssec:kms}

An extremely useful characterization of the kernel for a semigroup of matrices is known.
\begin{theorem}	\cite[Props. 1.11, 1.12]{HM}	\label{thm:minrank}  
Let $\SG$ be a finite semigroup of matrices. Then the kernel, $\K$, of $\SG$ is the set of matrices with minimal rank.
\end{theorem}
Suppose $\SG=\SG((C_1,C_2,\ldots,C_d))$ is a semigroup generated by binary stochastic matrices $C_i$.
Then $\SG$ is a finite semigroup with kernel $\mathcal{K} = \ms.X. \times \ms.G. \times \ms.Y.$, the Rees product structure of
Theorem \ref{thm:kernel}. Let $k$, $k'$ be elements of $\mathcal{K}$. 
Then with respect to the Rees product structure $k=\left( k_1,k_2,k_3\right) $ and $k'=\left( k'_1,k'_2 ,k'_3\right) $. 
We form the ideals $\mathcal{K}k$, $k'\mathcal{K}$, and $k'\mathcal{K}k$: \bigskip

\begin{enumerate}
\item $\mathcal{K}k = \ms.X. \times \ms.G. \times \left\{k_3\right\}$ is a minimal left ideal in $\mathcal{K}$ whose elements all have the same range, or nonzero columns as $k$. 
We call this type of semigroup a \textit{left group}.  \bigskip

\item $k'\mathcal{K} = \left\{k'_1\right\} \times \ms.G. \times \ms.Y.$ is a minimal right ideal in $\mathcal{K}$ whose elements all have the same partition of the vertices as $k'$. 
A block $\B_j$ in the partition can be assigned to each nonzero column of $k'$ by 
\[
\B_j =\{ i : k'_{i,j}=1\}
\]
A semigroup with the structure $k'\mathcal{K}$ is a \textit{right group}.  \bigskip

\item $k'\mathcal{K}k$ is the intersection of $k'\mathcal{K}$ and $\mathcal{K}k$. It is a maximal group in $\mathcal{K}$ 
(an $H$-class in the language of semigroups). It is best thought of as the set of functions specified by the partition of $k'$ and the range of $k$,
acting as a group of permutations on the range of $k$.
The idempotent of $k'\mathcal{K}k$ is the function which is the identity when restricted to the range of $k$.
\end{enumerate}
\subsection{Probability measures on finite semigroups}
Since we have a discrete finite set, a probability measure is given as a function on the elements. 
For matrices, the semigroup algebra is the algebra generated by the elements $w\in\SG$. In general, we consider formal
sums $\sum f(w)\,w$. The function $\mu$ defines a probability measure if 
$$0\le\mu(w)\le1\,,\forall w\in\SG\,, \text{ and  } \sum \mu(w)=1\ .$$
The corresponding element of the semigroup algebra is thus $\sum \mu(w)\,w$. The product of elements in the semigroup algebra yields
the convolution of the coefficient functions. Thus, for the convolution of two measures 
$\mu_1$ and $\mu_2$ we have
\begin{align*}
\sum_{w\in\SG} \mu_1*\mu_2(w)\,w&=\bigl(\sum_{w\in\SG} \mu_1(w)\,w\bigr)\bigl(\sum_{w'\in\SG} \mu_2(w')\,w'\bigr)\\
&=\sum_{w,w'\in\SG} \mu_1(w)\mu_2(w')\,ww'
\end{align*}
Hence the convolution powers, $\mu^{(n)}$ of a single measure $\mu=\mu^{(1)}$ satisfy
$$\sum_{w\in\SG} \mu^{(n)}(w)\,w=\bigl(\sum_{w\in\SG} \mu^{(1)}(w)\,w\bigr)^n$$
in the semigroup algebra. \bigskip

\begin{subsubsection}{\sl Invariant measures on the kernel}
We can consider the set of probability measures with support the kernel of a finite semigroup of matrices. Given such a measure $\mu$, it is 
\textit{idempotent} if $\mu*\mu=\mu$, i.e., it is idempotent with respect to convolution. An idempotent measure on a finite group must be the uniform distribution
on the group, its Haar measure. In general, we have 
\begin{theorem}\cite[Th. 2.8]{HM} \label{thm:idem}
An idempotent measure $\mu$ on a finite semigroup $\SG$ is supported on a completely simple subsemigroup, $\K'$ of $\SG$.
With Rees product decomposition $\K'=\ms.X'.\times\ms.G'.\times\ms.Y'.$, $\mu$ is a direct product of the 
form $\alpha\times\omega\times\beta$, where $\alpha$ is a measure on $\ms.X',.$ $\omega$ is Haar measure on $\ms.G',.$ and
$\beta$ is a measure on $\ms.Y'.$.
\end{theorem}
In our context, we will have the measure $\mu$ supported on the kernel $\K$ of $\SG$.
We identify $\ms.X.$ with the partitions of the kernel and $\ms.Y.$ with the ranges. The local groups are mutually isomorphic finite
groups with the Haar measure assigning $1/|G|$ to each element of the local group $G$. Thus, if $k\in\K$ has Rees product decomposition
$(k_1,k_2,k_3)$ we have
$$\mu(k)=\alpha(k_1)\beta(k_3)/|G|$$
where $|G|$ is the common cardinality of the local groups, $\alpha(k_1)$ is the $\alpha$-measure of the partition of $k$, and $\beta(k_3)$ is the
$\beta$-measure of the range of $k$.
\end{subsubsection} \bigskip

\begin{example} Here is an example with $n=6$. Take two functions $r=[451314]$, $b=[245631]$ and generate the semigroup $\SG$.
We find the elements of minimal rank, in this case it is 3.
The structure of the kernel is summarized in  a table with rows labelled by the partitions and columns labelled by the range classes. The entry is
the idempotent with the given partition and range. It is the identity for the local group of matrices with the given partition and range.
\begin{equation*}\label{eq:k6}
\begin{array}{c|c|c|c|c|}
&\{1, 3, 4\}&\{1, 4, 5\}&\{2, 3, 6\}&\{2, 5, 6\}\vstrut\cr
\hline
\{\{1, 2\}, \{3, 5\},\{4, 6\} \}& [1 1 3 4 3 4]& [1 1 5 4 5 4]&[2 2 3 6 3 6]&[2 2 5 6 5 6] \vstrut\cr
\hline
\{\{1, 6\},\{2, 4\},  \{3, 5\}\}& [1 4 3 4 3 1]&[1 4 5 4 5 1]&[6 2 3 2 3 6]& [6 2 5 2 5 6] \vstrut\cr
\hline
\end{array}
\end{equation*}
The kernel has 48 elements, the local groups being isomorphic to the symmetric group $S_3$.  \bigskip

Each cell consists of functions with the given partition, acting as permutations on the range class. They are given
by matrices whose nonzero columns are in the positions labelled by the range class. The entries in a nonzero column are in rows
labelled by the elements comprising the block of the partition mapping into the element labelling the column.
We label the partitions 
$$\P_1=\{\{1, 2\}, \{3, 5\},\{4, 6\} \}\,,\qquad \P_2=\{\{1, 6\},\{2, 4\},  \{3, 5\}\}$$
and the range classes
$$\R_1=\{1, 3, 4\}\,,\quad \R_2=\{1, 4, 5\}\,,\quad \R_3=\{2, 3, 6\}\,,\quad \R_4=\{2, 5, 6\}\ .$$
The local group with partition $\P_1$ and range $\R_3$, for example, are the idempotent $[2 2 3 6 3 6]$ noted in the above table and the functions
$$\{\,[6 6 3 2 3 2], [6 6 2 3 2 3], [3 3 6 2 6 2], [3 3 2 6 2 6], [2 2 6 3 6 3]\,\}$$
isomorphic to $S_3$ acting on the range class $\{2,3,6\}$. For the function $ [3 3 2 6 2 6]$, in the matrix semigroup we have correspondence
$$  [3 3 2 6 2 6]\  \longleftrightarrow\  \begin{bmatrix}
0& 0& 1& 0& 0& 0\\0& 0& 1& 0& 0& 0\\ 0& 1& 0& 0& 0& 0\\0& 0& 0& 0& 0& 1\\0& 1& 0& 0& 0& 0\\0& 0& 0& 0& 0& 1
\end{bmatrix}\ .
$$
\bigskip

The measure on the partitions is
$$\alpha=[2/3,1/3]$$
while that on the ranges is
$$\beta=[4/9,2/9,1/9,2/9]$$
as required by invariance.
\end{example} 
\end{section}

\begin{section}{Graphs, semigroups, and dynamical systems}
We start with a regular $d$-out directed graph on $n$ vertices with adjacency matrix $\A$. 
Number the vertices once and for all and identify vertex $i$ with $i$ and vice versa. \bigskip

Form the stochastic matrix $A=d^{-1}\,\A$.  We assume that $A$ is irreducible and aperiodic.
In other words, the graph is strongly connected and the limit
$$\lim_{m\to\infty}A^m=\Omega$$
exists and is a stochastic matrix with identical rows, the invariant distribution for the corresponding Markov chain, which we denote
by $\pi=[p_1,\ldots,p_n]$. The limiting matrix satisfies
$$\Omega=A\Omega=\Omega A=\Omega^2$$
so that the rows and columns are fixed by $A$, eigenvectors with eigenvalue 1. \bigskip

Decompose $A=\frac{1}{d} C_1+\cdots +\frac{1}{d} C_d$, into binary stochastic matrices, 
``colors", corresponding to the $d$ choices moving from a given vertex to another. Each coloring matrix $C_i$ is the matrix
of a function $f\in\F$ on the vertices as in the discussion in \S2. Let
$\SG=\SG((C_1,C_2,\ldots,C_d))$ be the semigroup generated by the matrices $C_i$, $1\le i\le d$. Then
we may write the decomposition of $A$ into colors in the form
$$A=\mu^{(1)}(1)C_1+\cdots +\mu^{(1)}(d) C_d=\sum_{w\in\SG} \mu^{(1)}(w)\,w$$
with $\mu^{(1)}$ a probability measure on $\SG$, thinking of the elements of $\SG$ as words $w$, strings from the alphabet generated
by $\{C_1,\ldots,C_d\}$. We have seen in the previous section that
\begin{equation}\label{eq:ces}
A^m=\sum_{w\in\SG}\mu^{(m)}(w)\,w
\end{equation}
where $\mu^{(m)}$ is the $m^{\rm th}$ convolution power of the measure $\mu^{(1)}$ on $\SG$. \bigskip

The main limit theorem is the following
\begin{theorem}\cite[Th. 2.13]{HM} Let the support of $\mu^{(1)}$ generate the semigroup $\SG$. Then the Ces\`aro limit
$$\lambda(w)=\lim_{N\to\infty} \frac{1}{N}\,\sum_{m=1}^N \mu^{(m)}(w)$$
exists for each $w\in\SG$. The measure $\lambda$ is concentrated on $\K$, the kernel of the semigroup $\SG$. It has the canonical decomposition
$$\lambda=\alpha \times \omega \times \beta$$
corresponding to the Rees product decomposition of $X\times G\times Y$ of $\K$.
\end{theorem}
\begin{proof}
See Appendix \ref{app:1} for a proof.
\end{proof}
 \bigskip

We use the notation $\avg {\cdot}.$ to denote averaging with respect to $\mu$ over random elements $K\in\K$. Forming
$N^{-1}\,\sum\limits_{m=1}^N$  on both sides of \eqref{eq:ces}, letting $N\to\infty$ we see immediately that
$$\Omega=\avg K.$$
the average element of the kernel. We wish to discover further properties of the kernel by developing a kernel hierarchy based 
on the constructions of \S2. \bigskip

Consider the semigroup generated by the matrices $\spow{C_i}{\l}$, corresponding to the action on $\l$-sets of vertices.
The map $w\to \spow{w}{\l}$ is a homomorphism of matrix semigroups, Theorem \ref{thm:hom}. From the above discussion, we
have the following.
\begin{proposition} Let $\phi:\SG\to \SG'$ be a homomorphism of matrix semigroups. Let
$$A_\phi=\mu^{(1)}(1)\phi(C_1)+\cdots +\mu^{(1)}(d) \phi(C_d)=\sum_{w\in\SG} \mu^{(1)}(w)\,\phi(w)$$
then the Ces\`aro limit 
$$\lim_{N\to\infty} \frac{1}{N}\,\sum_{m=1}^N A_\phi^{\,m}$$
exists and equals
$$\Omega_\phi=\avg \phi(K).$$
the average over the kernel $\K$ with respect to the measure $\alpha\times\omega\times\beta$ as for $\SG$.
\end{proposition}
\begin{proof} Checking that
$$A_\phi^{\,m}=\sum_{w\in\SG} \mu^{(m)}(w)\,\phi(w)$$
the result follows immediately.
\end{proof}
\begin{remark}
Especially for computations, e.g. with Maple or Mathematica, it is notable that we can use Abel limits instead of Ces\`aro. 
See Appendix \ref{app:2} for details.
\end{remark} \bigskip

We now have our main working tool.
\begin{proposition} For $1\le \l\le n$, define
$$A_\l=\frac{1}{d} \spow{C_1}{\l}+\cdots +\frac{1}{d} \spow{C_d}{\l}\ .$$
Then the Ces\`aro limit of the powers $A_\l^{\,m}$ exists and equals 
$$\Omega_\l=\avg \spow{K}{\l}.$$
the average taken over the kernel $\K$ of the semigroup $\SG$ generated by $\{C_1,\ldots,C_d\}$.
\end{proposition}
$\Omega_\l$ satisfies the relations
$$A_\l\Omega_\l=\Omega_\l A_\l=\Omega_\l^2=\Omega_\l\ .$$
\begin{proposition}
The rows of $\Omega_\l$ span the set of left eigenvectors of $A_\l$ for eigenvalue 1. The columns of $\Omega_\l$ span
the set of right eigenvectors of $A_\l$ for eigenvalue 1.
\end{proposition}
\begin{proof}
For a left invariant vector $v$, we have 
\begin{align*}
v=vA_\l=vA_\l^{\,m}=v\Omega_\l &&\text{ (taking the Ces\`aro limit)}
\end{align*}
Thus, $v$ is a linear combination of the rows of $\Omega_\l$. Similarly for right eigenvectors.
\end{proof}
Now, as $\Omega_\l$ is an idempotent, we have $\rk \Omega_\l=\tr \Omega_\l$ so
\begin{corollary} The dimension of the eigenspace of right/left invariant vectors at level $\l$ equals $\tr \Omega_\l$.
\end{corollary}
Let $r$ denote the rank of the kernel, the common rank of $k\in\K$. Then we see that $\Omega_\l$ vanishes for $\l>r$.  \bigskip

It is important to remark that $A_\l$ is in general \textit{substochastic}, i.e., it may have zero rows or rows that do not sum to one.
As noted in the Appendix \ref{app:2}, the Abel limits exist for the $A_\l$. They agree with the Ces\`aro limits.

\begin{subsection}{Hierarchy of kernels}
For $1\le\ell\le r$, we have a \textsl{hierarchy of kernels} $$\spow{\K}{\l}=\{\spow{k}{\l}\colon k\in\K\}\ .$$ 
Because we are working with all functions on a finite set, not only permutations, for the higher powers $\spow{F}{\l}$, $\l>1$, we will
have a row of zeros for any $\rI$ collapsed by the corresponding function $f$. At each level, we may introduce an additional state,
the \textit{collapsed state}, $\ttx$. 
We adjoin an additional row and column, labelled $\ttx$, to $\spow{F}{\l}$, denoting the extended matrix by $\spow{\dot F}{\l}$, and define 
$$ \spow{\dot{F}}{\l}_{\rI\rJ}=\begin{cases}1,&\text{ if } f(\rI)=\rJ\in\P_\l\\ 
1,&\text{ if } |f(\rI)|<|\rI|,\ \rJ=\ttx\\
1, &\text{ if } \rI=\rJ=\ttx\\
0,&\text{ otherwise}
\end{cases}
$$
In other words, the collapsed state $\ttx$, for any level $\l$, corresponds to a fixed point of the function. It is 
an absorbing state of the corresponding transition diagram. This way, we recover binary stochastic matrices at every level.
This may be called the \textit{augmented} matrix or augmented $\l^{\rm th}$ power. \bigskip

\begin{example}
With $f=[2344]$, we have
$$\spow{F}{2}= \left[ \begin {array}{cccccc} 0&0&0&1&0&0\\
0&0&0&0&1&0\\0&0&0&0&1&0\\0&0&0&0&0&1\\
0&0&0&0&0&1\\0&0&0&0&0&0\end {array} \right]\text{ and  }
\spow{\dot F}{2}= \left[ \begin {array}{ccccccc} 0&0&0&1&0&0&0\\
0&0&0&0&1&0&0\\
0&0&0&0&1&0&0\\0&0&0&0&0&1&0\\
0&0&0&0&0&1&0\\0&0&0&0&0&0&1\\0&0&0&0&0&0&1\end {array} \right]$$
\end{example}

\begin{example} Continuing with the kernel from \S3, at level 1, the partitions are
$$\P_1=\{\,\{1, 2\}, \{3, 5\},\{4, 6\}\, \}\,,\qquad \P_2=\{\,\{1, 6\},\{2, 4\},  \{3, 5\}\,\}$$
and the range classes
$$\R_1=\{1, 3, 4\}\,,\quad \R_2=\{1, 4, 5\}\,,\quad \R_3=\{2, 3, 6\}\,,\quad \R_4=\{2, 5, 6\}\ .$$
For convenience, we write out the level 2 indices:
$$
\begin{array}{rrrrrrrrrrrrrrr}
1&\, 2&\,3&\,4&\,5&\,6&\,7&\,8&\,9&\,10&\,11&\,12&\,13&\,14&\,15\\
12&\, 13&\,14&\,15&\,16&\,23&\,24&\,25&\,26&\,34&\,35&\,36&\,45&\,46&\,56
\end{array}
$$
 At level 2, we have
\begin{align*}
\P_1 &= \{\,\{1, 11, 14, \ttx\}, \{2, 4, 6, 8\}, \{3, 5, 7, 9\}, \{10, 12, 13, 15\}\,\}\\
\P_2 &= \{\,\{1, 3, 9, 14\}, \{2, 4, 12, 15\}, \{5, 7, 11, \ttx\}, \{6, 8, 10, 13\}\,\}
\end{align*}
$$                     \R_1= \{2, 3, 10, \ttx\}\,,\,\R_2= \{3, 4, 13, \ttx\}  \,,\,\R_3=\{6, 9, 12, \ttx\}   \,,\,          \R_4= \{8, 9, 15, \ttx\}$$
At level 3,
\begin{align*}
       \P_1&= \{\,\{5, 7, 8, 10, 11, 13, 14, 16\}, \{1, 2, 3, 4, 6, 9, 12, 15, 17, 18, 19, 20, \ttx\}\,\}\\
       \P_2&= \{\,\{1, 3, 5, 8, 13, 16, 18, 20\}, \{2, 4, 6, 7, 9, 10, 11, 12, 14, 15, 17, 19, \ttx\}\,\}
\end{align*}
$$  \R_1= \{5, \ttx\} \,,\quad                        \R_2= \{8, \ttx\}         \,,\quad                \R_3= \{13, \ttx\}        \,,\quad                 \R_4= \{16, \ttx\}$$
{\bf Note.} The level three labels are written out as part of this example continued at the end of \S4.
\end{example}

We describe the features of the higher-order kernels with the collapsed state included. 
\begin{enumerate}     
\item At level $\l$, a given range becomes the set of its $\l$-subsets with the collapsed state adjoined. Any $\l$-tuple not preserved by the
functions of the kernel, for a given partition, will map into the collapsed state.\medskip

\item For a given partition, the associated level-$\l$ partition is obtained as follows.
To get a block at level $\l$, take a subset of $\l$ blocks of the original partition. Form the cross product of those blocks. The corresponding
sets/labels comprise the level-$\l$ block. There will be an additional block for the collapsed state. Any $\l$-set containing two or more elements
from the same block will be collapsed. So a special block consists of all $\l$-sets collapsed in the partition and the collapsed state
$\ttx$. It is the block mapped into the $\ttx$ of each range class.\medskip

\item At level $r$, each range is its own label, along with the collapsed state. Each kernel element with a prescribed partition and range will map all
cross-sections, formed by taking the cross product of all $r$ blocks of the given partition, into the range set, a single label. All other $r$-sets
map into the collapsed state.\medskip

\item One recovers $\spow{\K}{\l}$ at levels $\l>1$ by dropping the $\ttx$'s from the range classes and dropping the entire
block containing $\ttx$ from each partition.
\end{enumerate}

\end{subsection}
\begin{subsection}{Kernel hierarchies}
First some conventions regarding notation. \par
We will use $u$ to denote a row vector of all ones, the dimension implicit according to context, 
with $\ud$ the all-ones column vector. Recall that $\pi$ is the invariant distribution for
the Markov chain generated by $A$.  \bigskip

Summarizing some basic features:
$$\pi A=\pi\,, \quad A\ud=\ud\,,\quad \Omega=\ud\pi\ .$$
We would like to extend the above relations to level $\l$. In other words, we are looking for invariant vectors, fixed by $A_\l$. Certainly
rows and columns of $\Omega_\l$ are candidates. Generally there are three classes of hierarchies that are of interest:
\begin{enumerate}     
\item We define $\pi_\l=u\Omega_\l$ and $u_{\l}^+=\Omega_\l \ud$.
\item Starting with $\pi_r$ a left invariant vector of $A_r$, we can construct a hierarchy as we did in \S2, using inclusion operators.
In this context, we will construct a hierarchy of right eigenvectors using  a special class of inclusion operators. We give details below.
\item We can solve for the right (resp. left) nullspace of $I-A_\l$ in each layer. Generally, the dimension of these nullspaces, $\tr\Omega_\l$,
will exceed 1. However, for level 1, by assumption $\tr\Omega_1=\tr \Omega=1$, the spaces spanned by $\ud$ and $\pi$ respectively.
\end{enumerate}
We will not discuss point (3) further in this paper, although it is of particular interest for $\l=2$. In general, we say that the system
has/is of $\l$-rank $r_\l$ if $\tr\Omega_\l=r_\l$. In the following sections, we will study approaches (1) and (2) in detail. We will use our
theory for several applications:  \bigskip

1. Finding the rank of the kernel.\par
2. Characterizing kernels which are right groups in terms of $u_2$.\par
3. Classifying all kernels of rank $n-1$ for the case of two colors.
\end{subsection}

\subsection{The fields $\pi_\l$ and $u_\l$}
First consider $u_2^+=\Omega_2\ud$. We have
$$u_2^+=\avg \spow{K}{2}\ud. $$
Now, for $k\in\K$, 
$$u_2(\rI)=(\spow{k}{2}\ud)_\rI=\sum_\rJ \spow{k}{2}_{\rI\rJ}=\sum_{j_1,j_2} (\spow{k}{2})_{i_1i_2,j_1j_2}$$
With $\rI=(i_1,i_2)$ fixed, this reads 1 just in case $\rI$ is $k$-preserved. So, this will equal 1 for 
all $k$ in a given row of the kernel, having the same partition, $\P$, say, if it is 1 for some such $k$. \bigskip

\begin{definition}  The vertices $i_1,\ldots,i_t$, $t\le r$, \textit{split} across partition $\P$ if they are in different blocks of $\P$. So they are
$k$-preserved for corresponding elements $k$ of the kernel. Otherwise, they \textit{collapse in $\P$}.
\end{definition}

Thus, 
$$u_2(\rI)=\sum_{(i_1,i_2) \text{ split across } \P}\alpha(\P)=\text{Prob}(\rI \text{ is split across a given partition})$$
with $\alpha$ the measure on the partitions of the kernel. The interpretation for $u_\l$ is similar:
$$u_\l(\rI)=\sum_{(i_1,i_2,\ldots,i_\l) \text{ split across } \P}\alpha(\P)=\text{Prob}(\rI \text{ is split across a given partition})$$
Note that for all $\l$, $1\le\l\le r$, an  $\l$-subset, $\rI$, of a range class has $u_\l(\rI)$ equal to 1. \bigskip

The expectation $\avg KK^\top.$ is closely related. Observe that
\begin{equation}\label{eq:kk}
 (kk^\top)_{i_1i_2}=\sum_j k_{i_1j}k_{i_2j}=\begin{cases} 1,& (i_1,i_2) \text{ collapse in }\P\\ 0,& \text{ otherwise}\end{cases}
\end{equation}
with $\P$ being the partition of $k$. We have, using the pair $(i,j)$ for the level-2 index $\rI$,  recalling the all-ones matrix $J$,
\begin{proposition}\label{prop:kkstar}
 The $(i,j)$ entry of $\avg J-KK^\top.$ equals $u_2(i,j)$.
\end{proposition}
In other words, if we form an $n\times n$ matrix with entries $u_2(i,j)$ we get exactly $\avg J-KK^\top.$.
It is of interest to note that $kk^\top=ee^\top$ if the idempotent $e\in X$ labels that row of the kernel, since it only depends on the partition.
So we can write as well
$$\avg J-KK^\top.=\avg J-EE^\top.$$
averaging, using the measure $\alpha$, over the idempotents in any single column of the kernel. \bigskip

For $\pi_\l$, the interpretation is not as straightforward except for $\l=1$, where $\pi_1=\pi$ is the invariant distribution for $A$ and
for $\pi_r$, which involves the range classes. We have
$$\pi_r=\avg u\spow{K}{r}.$$
And for $k\in\K$,
$$(u\spow{k}{r})_\rJ=\sum_\rI \spow{k}{r}_{\rI\rJ}$$
Since we are at level $r$, the $r$-tuple $\rI$ must be split across the partition, $\P$, of $k$, i.e., it must be $k$-preserved and map onto a range class, $\rJ$.
Every $r$-tuple split across $\P$ maps to $\rJ$ by every element in the same cell of the kernel as $k$. So for $\rJ$ fixed, averaging over the column
with $\rJ$ as range class yields the sum
$$\kappa=\sum_\P \alpha(\P)\,\prod_{\text{blocks }\B\text{ of }\P} |\B(\P)|$$
where for a given $\P$ we get the cardinality of the Cartesian product of the blocks comprising it.
Note that this number is independent of $\rJ$. Thus, when averaging over $\K$ we get $\kappa \,\beta(\rJ)$, the $\beta$-probability of the range class.
Thus,
\begin{proposition}
$\pi_r$ is supported on the range classes of the kernel. The corresponding distribution, normalized with total weight equal to one, is precisely
the measure $\beta$ on the range classes.
\end{proposition}

\begin{remark} The $u_\l$ vectors are scaled so that the maximum entry equals 1, the probability of a set being split across all of the partitions.
In particular, this is the case for each range class. Note that a zero entry in $u_2$ means that the corresponding pair appear in the same block for
every partition.\medskip

The $\pi_l$ vectors may be scaled as probability distributions, so that the sum of the entries equals 1.
\end{remark}

\subsection{Hierarchies by inclusion operators}
The $\pi_\l$ hierarchy works very similarly as for the function hierarchy in \S2, since $\pi_r$ is supported on the range classes, hence
is compatible for all colors $C_i$. Thus, in this context we define
$$\pi_{r-1}=\pi_r\EP{E}{r,r-1},\ldots,\pi_\l=\pi_{\l+1}\EP{E}{\l+1,\l}, \ldots$$
where $\pi_1$ will in general be a multiple of $\pi$, requiring scaling to a probability distribution. 
More explicitly, we may write the transition from level $\l+1$ to level $\l$ 
$$ \pi_\ell(I)=\sum_{\rI\subset \rJ}\pi_{\ell+1}(J) $$

We have $\pi_r$ as a left eigenvector of $A_r$ and
$C_i$-compatible for every coloring $C_i$. Then, proceeding inductively, we have, with $\pi_\l$ and $\pi_{\l-1}$, $C_i$-compatible, $1\le i\le d$,
$$\pi_{\l-1}\spow{C_i}{(\l-1)}=\pi_\l\EP{E}{\l,\l-1}\spow{C_i}{(\l-1)}=\pi_\l\spow{C_i}{\l}\EP{E}{\l,\l-1}$$
Now, average over $i$ to recover $A_{\l-1}$, $A_\l$, and using the invariance of $\pi_l$ for $A_\l$,
$$\pi_{\l-1}A_{\l-1}=\pi_\l\EP{E}{\l,\l-1}A_{\l-1}=\pi_\l A_{\l}\EP{E}{\l,\l-1}=\pi_\l\EP{E}{\l,\l-1}=\pi_{\l-1}$$
as required. \bigskip

For the $u_\l$ hierarchy, we need a useful theorem due originally to Friedman \cite{FRC}, reformulated in the semigroup setting in \cite{BuMu}. 
\begin{theorem} [\cite{BuMu, FRC}, version 1]\label{thm:bmf1}
Let $\B$ be a block in any partition of the kernel. Then
$$\pi(\B)=\sum_{i\in\B} p_i=\frac{1}{r}$$
where $p_i$ are the components of the invariant distribution $\pi$.
\end{theorem}
Now we have the $u_\l$-hierarchy.

\begin{theorem} Define inductively, starting from $u_r$,
$$ u_\l(I)=\sum_i p_i \,u_{\ell+1}(I\cup \{i\}) $$
Then $u_\l$ is a multiple of $\Omega_\l\ud$ for each $\l$, $1\le\l\le r$.
\end{theorem}
So this hierarchy is consistent with that constructed using $\Omega_\l$ at each level.
\begin{proof}
We have the components of $u_r$ as the $\alpha$-probabilities of the indices being split across partitions. Proceeding inductively,
assume we are at level $\l+1$. Use the stated formula to define $u_\l$:
$$u_\l(I)=\sum_i p_i \,u_{\ell+1}(I\cup \{i\}) \ .$$
Consider $\rI\in\P_\l$. To get a nonzero contribution from partition $\P$ at level $\l+1$, we must have $\rI\cup\{i\}$ split across $\P$.
The $\l$ indices of $\rI$ split across $\l$ of the blocks, leaving $r-\l$ blocks from which to choose $i$. 
By Friedman's Theorem, these contribute a total probability of $(r-\l)/r$. As this is true for any $\rI\in\P_\l$, 
we get $(r-\l)/r$ as a common factor times the probability that $\rI$ is split across $\P$, 
thus $u_\l$ is a multiple of $\Omega_\l\ud$ as required.
\end{proof}
The entries of the inclusion operators for the $u$-hierarchy are weighted according to $\pi$.
$$\EP{\Pi}{\l,m}=\begin{cases} {\scriptstyle\prod\limits_{j\in \rJ\setminus \rI}} p_j,&\text{ if } \P_\l\ni\rI\subset\rJ\in\P_m\\ 0,&\text{ otherwise}\end{cases}$$
Note that the $u$ hierarchy are naturally column vectors, so we write the hierarchy relations in the form
$$u_\l^+=\EP{\Pi}{\l,\l+1}u_{\ell+1}^+$$
For $n=5$, we have, for example
$$\EP{\Pi}{1,2}=\left[ \begin {array}{cccccccccc} 
p_2&p_3&p_4&p_5&0&0&0&0&0&0\\
p_1&0&0&0&p_3&p_4&p_5&0&0&0\\
0&p_1&0&0&p_2&0&0&p_4&p_5&0\\
0&0&p_1&0&0&p_2&0&p_3&0&p_5\\
0&0&0&p_1&0&0&p_2&0&p_3&p_4
\end {array} \right] $$
Reading the indices $p_i$ in a column tells you the label for that column, the set $\rJ$.
And for the transition from layer 4 to layer 1, composing one layer at a time, we have
$$\EP{\Pi}{1,4}=\frac{1}{3!}\,\EP{\Pi}{1,2}\EP{\Pi}{2,3}\EP{\Pi}{3,4}=
\left[ \begin {array}{ccccc} p_2p_3p_4&p_2p_3p_5&p_2p_4p_5&p_3p_4p_5&0\\
p_1p_3p_4&p_1p_3p_5&p_1p_4p_5&0&p_3p_4p_5\\
p_1p_2p_4&p_1p_2p_5&0&p_1p_4p_5&p_2p_4p_5\\
p_1p_2p_3&0&p_1p_2p_5&p_1p_3p_5&p_2p_3p_5\\
0&p_1p_2p_3&p_1p_2p_4&p_1p_3p_4&p_2p_3p_4\end {array} \right]
$$
where these act on column vectors.  \bigskip

\begin{example}
Consider the system at the end of \S\ref{sec:sgrpskrnls}. We have the partitions
$$\P_1=\{\{1, 2\}, \{3, 5\},\{4, 6\} \}\,,\qquad \P_2=\{\{1, 6\},\{2, 4\},  \{3, 5\}\}$$
and ranges
$$\R_1=\{1, 3, 4\}\,,\ \R_2=\{1, 4, 5\}\,,\ \R_3=\{2, 3, 6\}\,,\ \R_4=\{2, 5, 6\}$$
with 
$$\alpha=[1/3,2/3]\,,\qquad \text{ and  }\beta=[4/9,2/9,1/9,2/9]\ .$$
Now, for convenience, writing out the indices at level 3:
$$
\begin{matrix}
1&2&3&4&5&6&7&8&9&10\\
123&124&125&126&134&135&136&145&146&156\\
\strut\\
11&12&13&14&15&16&17&18&19&20\\
234&235&236&245&246&256&345&346&356&456
\end{matrix}
$$
we have
$$u_3=\bigl[\frac{2}{3}, 0, \frac{2}{3}, 0, 1, 0, \frac{1}{3}, 1, 0, \frac{1}{3}, \frac{1}{3}, 0, 1, \frac{1}{3}, 0, 1, 0, \frac{2}{3}, 0, \frac{2}{3}\bigr]    $$    
which can be checked directly using $\alpha$. For example, we have $\{1,2,3\}$ split across $\P_2$, but not $\P_1$, so
$u_3( \{1,2,3\}) = \alpha(\P_2)=2/3$ while $u_3(\{1,3,4\}=1$, as it is a range class.   \bigskip

Now, $\pi_3$ directly corresponds to $\beta$,
$$\pi_3= [0, 0, 0, 0, 4/9, 0, 0, 2/9, 0, 0, 0, 0, 1/9, 0, 0, 2/9, 0, 0, 0, 0]\ .$$
Note that for a kernel of rank $r$,
 if $\pi_r$ has a non-zero entry, then the corresponding entry in $u_r$ must equal 1; however, in general the converse does not hold. \bigskip

Rescaling $\pi_3$, the inclusion hierarchy yields
\begin{gather*}
  \pi_3=[0, 0, 0, 0, 4, 0, 0, 2, 0, 0, 0, 0, 1, 0, 0, 2, 0, 0, 0, 0]\\
         \pi_2=[0, 4, 6, 2, 0, 1, 0, 2, 3, 4, 0, 1, 2, 0, 2]\\
           \pi_1= [12, 6, 10, 12, 8, 6]=2\,[6,3,5,6,4,3]
\end{gather*}
For example, $\pi_2(\{2,6\})$, ninth entry from the left, has a weight of $3$, since $\{2,6\}$ is contained in both $\R_3$ and $\R_4$. And
$\pi_1=54\pi$.  \bigskip

Starting with unscaled $u_3$, we have the inclusion hierarchy
\begin{gather*}
u_3=  [2, 0, 2, 0, 3, 0, 1, 3, 0, 1, 1, 0, 3, 1, 0, 3, 0, 2, 0, 2]\\
u_2=\bigl[\frac{2}{3},1,1,1,\frac{1}{3},1,\frac{1}{3},1,1,1,0,1,1,\frac{2}{3},1 \bigr]\\
u_1= [1, 1, 1, 1, 1, 1]
\end{gather*}
For example, the first entry in $u_2$ comes from the fact that $\{1,2,3\}$ and $\{1,2,5\}$ are split in $\P_2$. 
We get
$$u_2(\{1,2\})=2\times \frac{5}{27}+2\times \frac{4}{27}=\frac{2}{3}$$
where in the first term, the  value $5/27=p_3$, the probability of the deleted vertex.
The $0$ entry in $u_2$ shows that the pair $\{3,5\}$ is collapsed in both partitions.
\end{example}
\end{section}
\begin{section}{Applications}
\begin{subsection}{Finding the rank of the kernel}
Once Friedman's theorem has been put into the current framework as in \cite{BuMu}, for this application, we
express it as follows:
\begin{theorem}[\cite{BuMu,FRC}, version 2]\label{thm:bmf2}
For $k\in\K$, let $\rho(k)$ denote the $0$-$1$ vector with support the range of $k$. Then
$$\pi k=\frac{1}{r}\rho(k)$$
where $r$ is the rank of the kernel.
\end{theorem}
\begin{corollary} For any $k\in\K$, we have
$$\pi kk^\top=\frac{1}{r}\,u$$
\end{corollary}
\begin{proof}
Consider $k\rho(k)^+$. In each row of $k$, the entry equal to one is precisely in one of the range spots supporting $\rho(k)^+$. So every row
contributes a 1. Now transpose.
\end{proof}
We may first calculate $u_2$ by finding $\Omega_2$, say as an Abel limit. Then, using Proposition \ref{prop:kkstar}, we have
$J-\avg KK^\top.$. Form an $n\times n$ matrix from the components of $u_2$ and subtract from $J$ to get $\avg KK^\top.$. Thus,
$$\pi\avg KK^\top.=\avg \pi KK^\top.=\frac{1}{r}\,u$$
by the Corollary above. \bigskip

\begin{example}
Continuing the example from the end of the previous section, we form the matrix $\hat u_2$ from $u_2$:
$$\begin{bmatrix}0& 2/3& 1& 1& 1& 1/3\\2/3& 0& 1& 1/3& 1& 1\\
1& 1& 0& 1& 0& 1\\1& 1/3& 1& 0& 1& 2/3\\ 1& 1& 0& 1& 0& 1\\ 1/3& 1& 1& 2/3& 1& 0\end{bmatrix}
$$
Note that the diagonal consists of all zeros and the matrix is symmetric. Now form $J-\hat u_2$ and multiply by $\pi$ on the left:
$$\frac{1}{27}\,[6,3,5,6,4,3]\begin{bmatrix} 1& 1/3& 0& 0& 0& 2/3\\1/3& 1& 0& 2/3& 0& 0\\
0& 0& 1& 0& 1& 0\\0& 2/3& 0& 1& 0& 1/3\\ 0& 0& 1& 0& 1& 0\\ 2/3& 0& 0& 1/3& 0& 1 \end{bmatrix}=
\bigl[\frac{1}{3},\frac{1}{3},\frac{1}{3},\frac{1}{3},\frac{1}{3},\frac{1}{3}\bigr]=\frac{1}{r}u $$
verifying that the rank is 3.
\end{example}
\end{subsection}
\begin{subsection}{Right groups}
Right groups, \S\ref{ssec:kms}, (2), are characterized by the kernel having one partition, so the measure $\alpha$ is equal to one on a single
partition. So, in Proposition \ref{prop:kkstar}, we have the single partition with any of the corresponding idempotents $e$ satisfying
$$u_2(i,j)=1-(ee^\top)_{ij}$$ 
Recalling equation \eqref{eq:kk}, we see that $u_2(\rI)$ takes only the values $0$ or $1$.  Conversely, if there are more than one partition,
some pair must be split in one but not the other. Then the value of $u_2$ on that pair would be positive and strictly less than one. Thus,
\begin{proposition} The kernel is a right group if and only if $u_2(\rI)$ takes only the values $0$ and $1$.
\end{proposition}
The pairs $\rI$ for which $u_2(\rI)=0$ are precisely those pairs in the same block. So in this case, the partition may be read off from $u_2$. \bigskip

\begin{example}
For the continued rank 3 example, we have seen that $u_2$ does not consist of only 0's and 1's. Hence the kernel is not a right group in agreement
with the known two partitions. Here is an example of a rank 4 right group. Take $r=[311644]$, $b=[544123]$. With $A=(1/2)(R+B)$,
the invariant distribution is $\pi=[4,1,3,4,2,2]/16$. And we find 
$$u_2= [1, 1, 1, 1, 1, 0, 1, 1, 1, 1, 1, 1, 1, 1, 0]$$
indicating that the kernel is a right group. Forming the corresponding matrix and subtracting from $J$, we calculate
$$\frac{1}{16}\,[4,1,3,4,2,2]\begin{bmatrix}
1& 0& 0& 0& 0& 0\\ 0& 1& 1& 0& 0& 0\\ 0& 1& 1& 0& 0& 0\\ 0& 0& 0& 1& 0& 0\\ 0& 0& 0& 0& 1& 1\\ 0& 0& 0& 0& 1& 1 \end{bmatrix}
=\bigl[\frac{1}{4},\frac{1}{4},\frac{1}{4},\frac{1}{4},\frac{1}{4},\frac{1}{4}  \bigr]=\frac{1}{4}\,u\ .$$
From the matrix $\hat u_2$ or this last, $J-\hat u_2$, we can immediately read off the blocks of the partition:
$$\P_1=\{\{1\},\{2,3\},\{4\},\{5,6\}\}$$
To find the range classes, we need $\pi_4$. We form 
$$A_4=(1/2)(\spow{R}{4}+\spow{B}{4})$$ and find a left invariant eigenvector, spanning the nullspace of $J-A_4$. We normalize to
$$\pi_4=\beta=[0, 0, 0, 1/4, 0, 0, 1/4, 2/4, 0, 0, 0, 0, 0, 0, 0]$$
with range classes 
$$\R_1=\{1,2,4,5 \}\,,\quad \R_2=\{1,3,4,5\}\,,\quad \R_3=\{1,3,4,6\}\ .$$
\end{example}
\end{subsection}
\begin{subsection}{Rank $n-1$ kernels}
Using Theorem \ref{thm:bmf1} and invariance of $\pi$, we will find the general form of $\pi$ for a kernel of rank $n-1$. 
It will be apparent that the kernel must be a right group. 
The up-down symmetry of Propositions \ref{prop:enminusone} and \ref{prop:nmo} will give us the measure $\beta$ and the range classes.
We consider the case of two colors only, typically denoting $C_1=R$, for red, and $C_2=B$ for blue. \bigskip

Consider a partition in the kernel. We have $n-1$ blocks. So two vertices are
in one block, with the rest in singletons. Number the doubleton pair vertices 1 and 2. Thus,
$$\P=\{\{1,2\},\{3\},\ldots,\{n\}\}\ .$$
By Theorem \ref{thm:bmf1}, each block has $\pi$-measure $1/(n-1)$.  Thus,
$$\pi=[\frac{1}{n-1}-q,q,\frac{1}{n-1},\frac{1}{n-1},\ldots,\frac{1}{n-1}]$$
for some $0<q<\frac{1}{n-1}$. In particular, $p_1$ and $p_2$ are strictly less than $\frac{1}{n-1}$.\bigskip

Since the graph is 2-out regular, the total in-degree is $2n$. The main point is to consider vertices of in-degree 3, equivalently in-degree 1.
The invariance of $\pi$ gives the equation
\begin{equation}\label{eq:inv}
2p_i=\sum_{j\to i} p_j
\end{equation}

at each vertex $i$. We make a series of observations: \bigskip

1. There are no vertices of degree exceeding 3, since eq. \eqref{eq:inv} yields
$$2p_i\ge \frac{1}{n-1}-q+q+2\frac{1}{n-1}=\frac{3}{n-1}$$
which possibility is eliminated. \bigskip

2. Not all vertices are of degree 2, the doubly stochastic case, since $q<\frac{1}{n-1}$. So there is at least one vertex of degree 3, correspondingly,
one of indegree 1. Let $s$ denote a vertex of in-degree 1. Then if $j\to s$, $p_j=2p_s$. So $s$ must be 1 or 2. Say it's vertex 1. There are two cases: \medskip

a. $2\to1$. In this case, 
$$q=2(\frac{1}{n-1}-q)\qquad\Rightarrow\qquad p_2=q=\frac{2/3}{n-1} \,,\text{ and  }p_1=\frac{1/3}{n-1}\ .$$
Note that $2$ has in-degree at least two, since incoming $j$ would have $p_j$ equal to $\frac{4/3}{n-1}$. \medskip

b. $j\to 1$, where $p_j=\frac{1}{n-1}$. So $p_1=p_2=\frac{1/2}{n-1}$ are equally likely. Both $1$ and $2$ have in-degree 1 in this case and neither
maps to the other.\bigskip

This gives us two possibilities. Now consider the coloring matrices $R$ and $B$. 
There are two ways to generate a kernel of rank $n-1$. In one case, $R$, say, is a permutation and $B$ has rank $n-1$. In the other case,
they both are rank $n-1$, so are already in the kernel, and they generate it. \bigskip

Note that lumping vertices 1 and 2, we would have $n-1$ vertices and a partition of all singletons, corresponding to a doubly stochastic system.
Now, to get case a., start with a doubly stochastic graph where
exactly one vertex, 1,  has a loop, from which we will create an edge from $2\to1$. Split it into two vertices keeping the in-degree of vertex
1 equal to one, as in Figure~1. In this case, one of the colorings is a permutation and the other has rank $n-1$. \bigskip

Case b. comes from a doubly stochastic system with no loops by splitting vertex 1 into two vertices of in-degree one, as in Figure~2.
The result is two rank $n-1$ generators.  \bigskip

Here is another way to look at these constructions. Recall the notation for a function, \S\ref{sssec:not}, using the ordered range to describe it.
Since we want to start with $n-1$ vertices and add a new vertex labelled 2, we will start with vertices labelled 1 through $n-1$, increase all of the
labels by 1, losing vertex 1. Then we will recover vertex 1 according to which case we want.  \bigskip

\begin{figure}[ht!]
  \centering
 \subfloat[Doubly stochastic with a loop]{\scalebox{.70}{\includegraphics{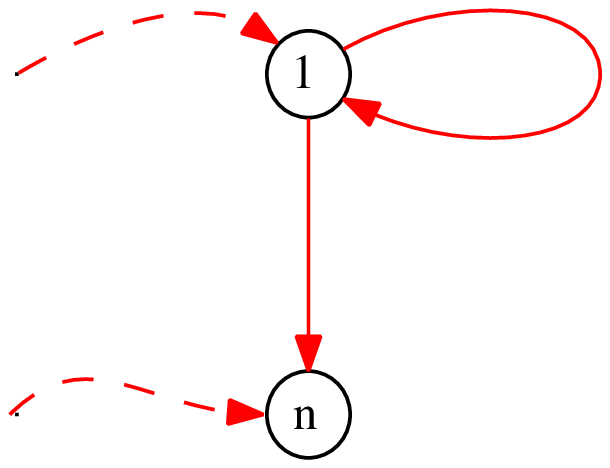}}}
  \subfloat[Rank $n-1$]{\scalebox{.70}{\includegraphics{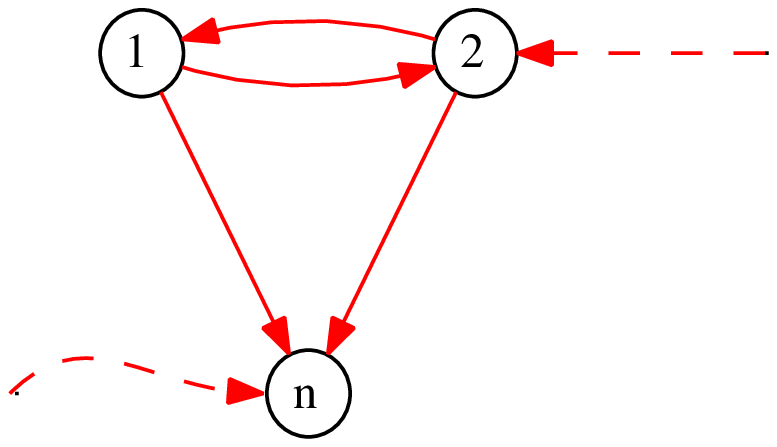}}}
  \caption{Triangular splitting}
\end{figure}

\begin{figure}[ht!]
  \centering
  \subfloat[Doubly stoch. no loop]{\scalebox{.70}{\includegraphics{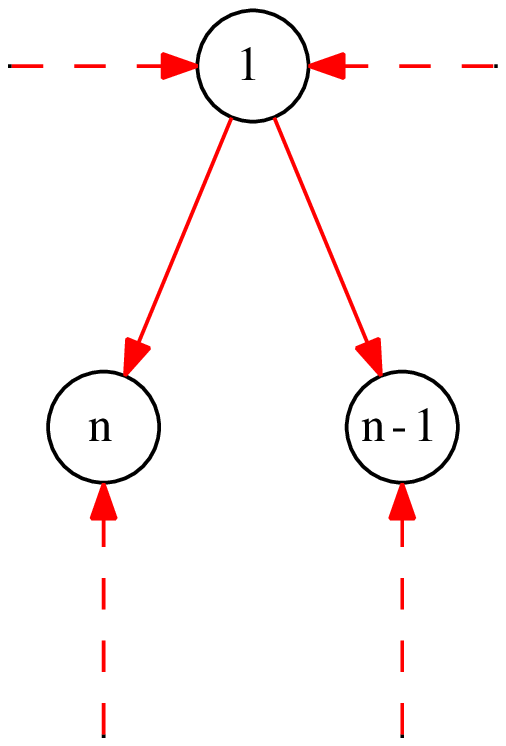}}} \phantom{-----------}               
  \subfloat[Rank $n\!-\!1$]{\scalebox{.70}{\includegraphics{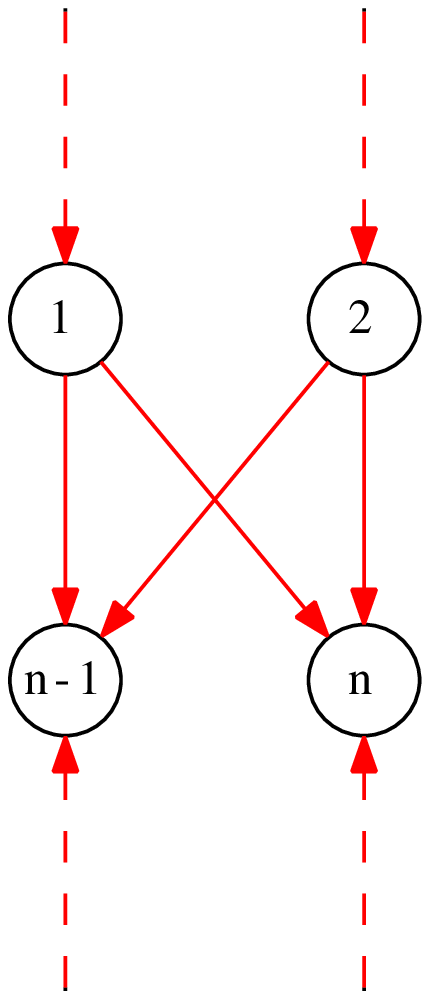}}}
  \caption{Cross splitting}
\end{figure}

Given $r=[r_1r_2\ldots r_{n-1}]$ and $b=[b_1b_2\ldots b_{n-1}]$, two permutations, with the matrices $R+B=2A$, the adjacency matrix for a doubly
stochastic, 2-out regular graph. Case a. means that vertex 1 is the unique vertex with a  loop, case b. holds when there are no loops. Then: \bigskip

1. Write the labels for $r$ and $b$ as a 2-rowed block. Increase all labels by 1.\medskip

2. Duplicate the first column.\medskip

3. For case a., change the ``2" coming from the loop into a ``1", creating a 2-cycle $1\leftrightarrow2$.
 In case b., pick any ``2" and change it to a ``1".  \bigskip

\begin{example}
1. For case a., consider $r=[162345]$, $b=[345612]$. The procedure looks like this:
$$\begin{matrix} 162345 \\ 345612 \end{matrix}\quad\to\quad \begin{matrix} 273456 \\ 456723 \end{matrix}
\quad\to\quad \begin{matrix} 2273456 \\ 4456723 \end{matrix}\quad\to\quad \begin{matrix} 2173456 \\ 4456723 \end{matrix}$$
with $r$ remaining a permutation. We have
$$A= \left[ \begin {array}{ccccccc} 0&1/2&0&1/2&0&0&0\\
1/2&0&0&1/2&0&0&0\\
0&0&0&0&1/2&0&1/2\\
0&0&1/2&0&0&1/2&0\\
0&0&0&1/2&0&0&1/2\\
0&1/2&0&0&1/2&0&0\\
0&0&1/2&0&0&1/2&0\end {array} \right]$$
with invariant distribution
$$\pi=[1/18,1/9,1/6,1/6,1/6,1/6,1/6]\ .$$

2. Here is an example for case b. Start with 3 vertices, $r=[312]$, $b=[231]$.  And we have:
$$\begin{matrix} 312 \\231 \end{matrix}\quad\to\quad\begin{matrix} 423 \\342 \end{matrix}
\quad\to\quad \begin{matrix} 4423 \\3342 \end{matrix}\quad\to\quad \begin{matrix} 4423 \\3341 \end{matrix}$$
with both $r$ and $b$ in the kernel. Here,
$$A= \left[ \begin {array}{ccccccc} 0&0&1/2&1/2\\0&0&1/2&1/2\\
0&1/2&0&1/2\\1/2&0&1/2&0\end {array} \right]$$
with invariant distribution
$$\pi=[1/6,1/6,1/3,1/3]\ .$$
\end{example}

Now, we apply the up-down symmetry to compute $\pi_r$. From Proposition \ref{prop:enminusone}, we have, up to a scale factor $c$,
\begin{align*}
c\,\pi_r&=\left[\frac{1}{n-1}-q,q,\frac{1}{n-1},\ldots,\frac{1}{n-1}\right]\,(\frac{1}{n-1}J-I')\\
&=\frac{1}{n-1}\,u-\left[\frac{1}{n-1},\ldots,\frac{1}{n-1},q,\frac{1}{n-1}-q\right]\\
&=[0,\ldots,0,\frac{1}{n-1}-q,q]
\end{align*}
So, in case a., where, only 1 has in-degree one, we have
$$\beta=[0,\ldots,0,1/3,2/3]$$
and when both have in-degree one, we have
$$\beta=[0,\ldots,0,1/2,1/2]\ .$$
In both cases, $\beta$ is concentrated on two range classes corresponding to the fact that 1 and 2 are collapsed in the kernel:
$$ \R_1=\{1,3,4,5,\ldots,n\}\qquad\text{ and  } \qquad \R_2=\{2,3,4,5,\ldots,n\}\ .$$
We can also write down $u_2$ immediately. Since all pairs except $\{1,2\}$ split, we have
$$u_2=[0,1,1,1,\ldots,1]\ .$$
The hierarchy operator yields $u_1$ with all components equal to $\frac{n-2}{n-1}$, which scales correctly to $u$. For example, for $j\ge3$,
$$(u_1)_j=p_1+p_2+\frac{1}{n-1}\,\sum_{3\le k\ne j}p_k=\frac{1}{n-1}+\frac{n-3}{n-1}=\frac{n-2}{n-1}$$
with the 1 and 2 components yielding $\frac{n-2}{n-1}$ immediately.
\end{subsection}
\end{section}
\begin{section}{Conclusion}
In this work, we have presented techniques for analyzing the structure of a completely simple semigroup in conjunction with idempotent measures
that it supports. Extending the action of a function on a set from points to subsets, a hierarchy of matrices and corresponding vectors is constructed
according to the sets which are mapped one-to-one by the given function. The mappings from a function to its action on $\l$-subsets of a given
$n$-set yield  homomorphisms of the transformation semigroup of functions acting on the set.  \bigskip

In the context of colorings of digraphs, decomposing the adjacency matrix into binary stochastic matrices corresponding to functions on the vertices,
the associated hierarchy constructions yield a hierarchy of kernels, one at each level, as well as sequences of vectors extending the left and right invariant
vectors of the corresponding Markov chain on the vertices. \bigskip

These topics extend in several directions. A natural setting is to study tensor hierarchies corresponding to tensor powers and 
various restricted tensor products, such as symmetric or antisymmetric tensors. Another direction is to study level 2 representations in depth. 
These  techniques allow one to study many features of permutation groups, cf. \cite{CA}, in the setting of transformation semigroups.
\end{section}
\begin{appendix}
\section{Convergence of convolution powers of a measure on a finite semigroup}\label{app:1}
For probability measures $\mu_1$, $\mu_2$ on the finite semigroup $\SG$, define the norm
$$\|\mu_1-\mu_2\|=\sup_{w\in\SG} |\mu_1(w)-\mu_2(w)|$$
the usual sup norm, in this case the same as the variational norm. \bigskip

\begin{theorem}\cite[Th. 2.13]{HM} \label{thm:conv}\medskip

Given a probability measure $\mu$ on $\SG$ whose support generates $\SG$. Let $\mu^{(k)}$ denote the $k^{\rm th}$ convolution power of $\mu$.
Then \bigskip

1. The sequence $\displaystyle \mu_n=\frac{1}{n} \sum\limits^n_{k=1} \mu^{(k)}$ converges to a probability measure $\lambda$ which satisfies the convolution
relations 
$$\lambda= \lambda *\lambda = \mu *\lambda = \lambda * \mu\ .$$

2. The support of $\lambda$ is the kernel $\K$ of the semigroup $\SG$. Let the Rees product decomposition of $\K$ be
$\ms.X.\times\ms.G.\times\ms.Y.$. Then $\lambda$ has the product decomposition
$$\lambda=\alpha\times\omega\times\beta$$
where $\alpha$ is an $\SG$-invariant measure on $\ms.X.$, $\omega$ is Haar measure on $\ms.G.$, 
and $\beta$ is an $\SG$-invariant measure on $\ms.Y.$.
\end{theorem}
We will give the proof of \#1 in detail.
\begin{proof}
We have
\[
\begin{aligned}
       \mu * \mu_n &= \mu * \bigl( \frac{1}{n}
          \sum^n_{k=1} \mu^{(k)} \bigr)                            \\
       &= \frac{1}{n} \sum^n_{k=1} \mu^{(k+1)}                   \\
       &= \frac{1}{n} (\mu^{(2)} + \mu^{(3)} + \ldots  + \mu^{(n+1)})\\
&=\mu_n*\mu
\end{aligned}
\]
So 
$$\mu_n = \mu * \mu_n + \frac{1}{n} [\mu - \mu^{(n+1)}] \quad 
\text{ and  }\quad \|\mu_n - \mu*\mu_n \| = \frac{1}{n} \|\mu-\mu^{(n+1)}\|\ .$$
and similarly for $\mu_n*\mu$.
Thus,
\[
\begin{aligned}
 \lim_{n\to\infty} \|\mu_n - \mu*\mu_n\| &= \lim_{n\to\infty} \frac{1}{n}\|\mu-\mu^{(n+1)}\|  \\
       &\le \lim_{n\to\infty} \frac{1}{n} \bigl(\|\mu\| + \|\mu^{(n+1)}\| \bigr)                  \\
       &= 0.
\end{aligned}
\]
Let $\lambda$ be any limit point of $(\mu_n)$, converging along the subsequence $\{n_i\}$. Now
\[
     \mu_{n_i}(w) = \mu*\mu_{n_i}(w) + \frac{1}{n_i} [\mu(w) - \mu^{(n_i+1)}(w)] \ .
\]
with a similar equation for  $\mu_{n_i}*\mu$. Letting $n_i\to\infty$, yields $$\lambda = \mu * \lambda =\lambda*\mu $$
hence
$$\mu * \lambda = \mu^{(2)} * \lambda = \lambda  \quad  \text{ and } \quad \lambda = \mu^{(k)} * \lambda = \lambda * \mu^{(k)} \text{, for all }k>0\ .$$
Then
\[
\begin{aligned}
       \mu_n * \lambda &= \left( \frac{1}{n} \sum^n_{k=1} \mu^{(k)} \right) * \lambda                             \\
       &= \frac{1}{n} \sum^n_{k=1}\mu^{(k)} * \lambda                                     \\
       &= \frac{1}{n} (n\lambda) = \lambda.
\end{aligned}
\]
Suppose $(\mu^{(n)})$ has another limit point $\lambda'$, where $\mu_{m_i} \to \lambda'$. Then
$$\forall n,\ \lambda = \mu^{(n)} * \lambda \Rightarrow\, \lambda = \mu_{m_i} * \lambda  $$

and letting $m_i\to\infty$, $$\lambda = \lambda' * \lambda\ .$$
 Similarly,   
$$\forall n,\ \lambda' = \lambda' * \mu^{(n)} \Rightarrow\, \lambda' = \lambda' * \mu_{n_i} $$
and letting $n_i\to\infty$,
 $$\lambda' = \lambda' * \lambda\ .$$
Thus, $\lambda' = \lambda$. 
We have $\lambda * \lambda = \lambda$, an idempotent measure on $\SG$. Theorem \ref{thm:idem} says that $\lambda$
has support the kernel of $\SG$, a completely simple semigroup with a product structure
$\ms.X. \times \ms.G. \times \ms.Y.$. Hence $\lambda$ has the stated form.
\end{proof}

\section{Abel limits}\label{app:2}
We are  working with stochastic and substochastic matrices. We give the details here on the
existence of Abel limits. Let $P$ be a substochastic matrix. The entries $p_{ij}$ satisfy
$0\le p_{ij}\le1$ with $\sum_jp_{ij}\le1$. Inductively, $P^n$ is substochastic for every integer $n>0$, with
$P^n$ stochastic if $P$ is.

\begin{proposition}     
Let $P$ be substochastic. The Abel limit
$$\Omega=\lim_{s\uparrow1}\, (1-s)(I-sP)^{-1}$$
exists and satisfies
$$ \Omega=\Omega^2=P\Omega=\Omega P$$
\end{proposition}
\begin{proof}     
We take $0<s<1$ throughout. \bigskip

For each $i,j$, the matrix elements $\avg e_i,P^ne_j.$ are uniformly bounded by 1. Denote
$$Q(s)=(1-s)(I-sP)^{-1}$$
So we have
$$\avg e_i,Q(s)e_j.=(1-s)\sum_{n=0}^\infty s^n\avg e_i,P^ne_j.\le (1-s)\sum_{n=0}^\infty s^n=1$$
so the matrix elements of $Q(s)$ are bounded by 1 uniformly in $s$ as well. Now take any sequence $\{s_j\}$,
$s_j \uparrow 1$. Take a further subsequence $\{s_{j11}\}$ along which the $11$ matrix elements converge.
Continuing with a diagonalization procedure going successively through the matrix elements, we
have a subsequence $\{s'\}$ along which all of the matrix elements converge, i.e., along which
$Q(s')$ converges. Call the limit $\Omega$. \bigskip

Writing $(I-sP)^{-1}-I=sP(I-sP)^{-1}$, multiplying by $1-s$ and letting $s \uparrow 1$ along $s'$ yields
$$\Omega=P\Omega=\Omega P$$
writing the $P$ on the other side for this last equality. Now, $s\Omega=sP\Omega$ 
implies $(1-s)\Omega=(I-sP)\Omega$. Similarly,  $(1-s)\Omega=\Omega(I-sP)$, so 
\begin{equation}\label{eq:abel} Q(s)\Omega=\Omega Q(s)=\Omega \end{equation}
Taking limits along $s'$ yields $\Omega^2=\Omega$. \bigskip

For the limit to exist, we check that $\Omega$ is the only limit point of $Q(s)$. From \eqref{eq:abel},
if $\Omega_1$ is any limit point of $Q(s)$, $\Omega_1\Omega=\Omega\Omega_1=\Omega$. 
Interchanging r\^oles of $\Omega_1$ and $\Omega$ yields 
$\Omega_1\Omega=\Omega\Omega_1=\Omega_1$ as well, i.e., $\Omega=\Omega_1$.\hfill\qedhere
\end{proof}

If $P$ is stochastic, it has $u^+$ as a nontrivial fixed point. In general we have

\begin{corollary}  
Let $\Omega=\lim\limits_{s\uparrow1}\, (1-s)(I-sP)^{-1}$ be the Abel limit of the powers of $P$.  \bigskip

$P$ has a nontrivial fixed point if and only if $\Omega\ne0$. \bigskip

\end{corollary}
\begin{proof}  
If $\Omega\ne0$, then $P\Omega=\Omega$ shows that any nonzero column of $\Omega$ is a nontrivial fixed point.
On the other hand, if $Pv=v\ne0$, then, as in the above proof, $Q(s)v=v$ and hence $\Omega v=v$ shows that $\Omega\ne0$.\hfill\qedhere
\end{proof}
Note that since $\Omega$ is a projection, its rank is the dimension of the space of fixed points.
\end{appendix}

\end{document}